\documentclass{amsart}
\usepackage{pstricks,pst-node,amsmath,latexsym,amssymb}
\newcommand{\Sp}{\mathop{\mathrm{Sp}}}
\newtheorem{theorem}{Theorem}
\newtheorem{corollary}[theorem]{Corollary}
\newcommand{\Irr}{\mathop{\mathrm{Irr}}}
\newcommand{\PGL}{\mathop{\mathrm{PGL}}}

\begin{document}

\title[]{A remark on the permutation representations afforded by the embeddings of $\mathrm{O}_{2m}^\pm(2^f)$ in $\mathrm{Sp}_{2m}(2^f)$}

\author[S.~Guest]{Simon Guest}
\address{Simon Guest, Mathematics, University of Southampton, Highfield, SO17 1BJ, United Kingdom}\email{s.d.guest@soton.ac.uk}

\author[A.~Previtali]{Andrea Previtali}

\author[P.~Spiga]{Pablo Spiga}
\address{
Dipartimento di Matematica e Applicazioni, University of Milano-Bicocca,\newline
Via Cozzi 53, 20125 Milano, Italy}\email{andrea.previtali@unimib.it, pablo.spiga@unimib.it}

\thanks{Address correspondence to P. Spiga,
E-mail: pablo.spiga@unimib.it}


\begin{abstract}
We show that the permutation module over $\mathbb{C}$ afforded by the action of $\mathrm{Sp}_{2m}(2^f)$ on its natural module is isomorphic to the permutation module over $\mathbb{C}$ afforded by the action of $\Sp_{2m}(2^f)$ on the union of the right cosets of $\mathrm{O}_{2m}^+(2^f)$ and $\mathrm{O}_{2m}^-(2^f)$.
\end{abstract}
\maketitle

\section{Introduction}\label{introduction}
That a given finite group can have rather different permutation representations affording the same permutation character was shown by Helmut Wielandt in~$1979$. For instance, the actions of the projective general linear group $\PGL_d(q)$ on the projective points and on the projective hyperplanes afford the same permutation character, but these actions are not equivalent when $d\geq 3$. A more interesting example is offered by the Mathieu group $M_{23}$. Here we have two primitive permutation representations of degree $253$ affording the same permutation character, but with non-isomorphic point stabilizers (see~\cite[p.~$71$]{ATLAS}).

Establishing  which properties are shared by  permutation representations of a finite group $G$ with the same permutation character has been the subject of considerable interest. For instance, it was conjectured by Wielandt~\cite{W} that, if $G$ admits two permutation representations on $\Omega_1$ and $\Omega_2$ that afford the same permutation character, and if $G$ acts primitively on $\Omega_1$, then $G$ acts primitively  on $\Omega_2$. This conjecture was first reduced to the case that $G$ is almost simple by F\"{o}rster and Kov\'acs~\cite{FK} and then it was solved (in the negative) by Guralnick and Saxl~\cite{GS}. Some more recent investigations on primitive permutation representations and their permutation characters can be found in~\cite{P}.

In this paper we construct two considerably different  permutation representations of the symplectic group that afford the same permutation character. We let $q$ be a power of $2$,  $G$ be  the finite symplectic group $\Sp_{2m}(q)$, $V$ be the $2m$-dimensional natural module for  $\Sp_{2m}(q)$ over the field $\mathbb{F}_q$ of $q$ elements, and $\pi$ be the complex permutation character for the action (by matrix multiplication) of $G$ on $V$. Since $q$ is even, the orthogonal groups $\mathrm{O}_{2m}^+(q)$ and $\mathrm{O}_{2m}^-(q)$ are maximal subgroups of $G$ (see~\cite{D}). For $\varepsilon\in \{+,-\}$, we let $\Omega^\varepsilon$ denote the set of right cosets of $\mathrm{O}_{2m}^\varepsilon(q)$ in $G$, and we let $\pi^\varepsilon$ denote the permutation character for the action of $G$ on $\Omega^\varepsilon$.

\begin{theorem}\label{thrm}
The $\mathbb{C}G$-modules $\mathbb{C}V$ and $\mathbb{C}\Omega^+\oplus\mathbb{C}\Omega^-$ are isomorphic. That is, $\pi=\pi^++\pi^-$.
\end{theorem}
We find this behaviour quite peculiar considering that the $G$-sets $V$ and $\Omega^+\cup \Omega^-$ are rather different. For instance, $G$ has two orbits of size $1$ and $q^{2m}-1$ on $V$, and has two orbits of size $q^m(q^m+1)/2$ and $q^{m}(q^m-1)/2$ on $\Omega^+\cup \Omega^-$. Moreover, the action of $G$ on both $\Omega^+$ and $\Omega^-$ is primitive, but the action of $G$ on $V\setminus \{0\}$ is not when $q>2$.

\section{Proof of Theorem~\ref{thrm}}\label{sec1}

Inglis~\cite[Theorem~$1$]{I} shows that the orbitals of the two orthogonal subgroups are self-paired, hence the characters $\pi^+$ and $\pi^-$ are multiplicity-free
(see \cite[\S 2.7]{C}). We will use this fact in our proof of Theorem~\ref{thrm}.

\begin{proof}[Proof of Theorem~\ref{thrm}] Let $1$ denote the principal character of $G$. Observe that $\pi=1+\pi_0$, where $\pi_0$ is the permutation character for the transitive action of $G$ on $V\setminus\{0\}$. In particular, for $v\in V\setminus\{0\}$, we have $\pi_0=1_{G_v}^G$, where $G_v$ is the stabilizer of $v$ in $G$. Frobenius reciprocity implies that $\langle\pi_0,\pi_0\rangle=\langle \pi_0|_{G_v},1\rangle$, and this equals the number of orbits of $G_v$ on $V\setminus\{0\}$. We claim that $G_v$ has $2q-1$ orbits on $V\setminus\{0\}$. More precisely we show that, given $w\in v^\perp\setminus\langle v\rangle$ and $w'\in V\setminus v^\perp$, the elements $\lambda v$ (for $\lambda\in \mathbb{F}_q\setminus\{0\}$),  $w$, and $\lambda w'$ (for $\lambda\in\mathbb{F}_q\setminus\{0\}$) are representatives for the orbits of $G_v$ on $V\setminus\{0\}$. Since $G_v$ fixes $v$ and preserves the bilinear form $(\,,\,)$, these elements are in distinct $G_v$-orbits. Let $u\in V\setminus\{0\}$. If $u\in \langle v\rangle$, then $u=\lambda v$ for some $\lambda\neq 0$, and hence there is nothing to prove. Let $w_0=w$ if $(v,u)=0$, and let $w_0=\frac{(v,u)}{(v,w')} w'$ if $(v,u)\neq 0$. By construction, the  $2$-spaces $\langle v,u\rangle$ and $\langle v,w_0\rangle$ are isometric and they admit an isometry $f$ such that $v^f=v$ and $u^f=w_0$. By Witt's Lemma~\cite[Proposition~$2.1.6$]{KL}, $f$ extends to an isometry $g$ of $V$. Thus $g\in G_v$ and $u^g=w_0$, which proves our claim.  Therefore, we have
\begin{equation}\label{eq1}
\langle \pi_0,\pi_0\rangle=2q-1.
\end{equation}

Next we need to refine the information in~\eqref{eq1}. Let $P$ be the stabilizer of the $1$-subspace $\langle v\rangle$ in $G$. Then $P$ is a maximal parabolic subgroup of $G$ and $P/G_v$ is cyclic of order $q-1$. Write $\eta=1_{G_v}^P$ and observe that $\eta=\sum_{\zeta\in \Irr(P/G_v)}\zeta$, where by abuse of terminology we identify
the characters  of $P/G_v$ with the characters of $P$ containing $G_v$ in the kernel. Thus
\[\pi_0=1_{G_v}^G=(1_{G_v}^P)_P^G=\eta_P^G=\sum_{\zeta\in \Irr(P/G_v)}\zeta_P^G.\]
Since every character of $G$ is real-valued~\cite{Gow}, we must have 
\[(\overline{\zeta})_P^G=\overline{\zeta_P^G}=\zeta_P^G,\] 
where $\overline{x}$ denotes the complex conjugate of $x \in \mathbb{C}$. Let $\mathcal{S}$ be a set of representatives, up-to-complex conjugation, of the non-trivial characters of $\Irr(P/G_v)$. Since $|P/G_v|=q-1$ is odd, we see that $|\mathcal{S}|=q/2-1$. We have
\[\pi_0=1_P^G+2\sum_{\zeta\in \mathcal{S}}\zeta_P^G.\]
If we write $\pi'=\sum_{\zeta\in\mathcal{S}}\zeta_P^G$, then we have $\pi_0=1_P^G+2\pi'$.

Since $1_P^G$ is the permutation character of the rank $3$ action of $G$ on the $1$-dimensional subspaces of $V$, we have $1_P^G=1+\chi^++\chi^-$ for some distinct non-trivial irreducible characters $\chi^+$ and $\chi^-$ of $G$. Let $\Gamma$ be the graph  with vertex set the $1$-subspaces
of $V$ and edge sets $\{\langle v\rangle, \langle w\rangle\}$ whenever $v\perp w$. Observe that $\Gamma$ is strongly regular with parameters
\[\left(\frac{q^{2m}-1}{q-1}, \frac{q^{2m-1}-q}{q-1},\frac{q^{2m-2}-1}{q-1}-2, \frac{q^{2m-2}-1}{q-1}\right).\]
Hence the eigenvalues of $\Gamma$ have multiplicity
$\frac{1}{2}\left(\frac{q^{2m}-q}{q-1}-q^m\right)$ and $\frac{1}{2}\left(\frac{q^{2m}-q}{q-1}+q^m\right)$
(see \cite[p. 27]{CV}).

Interchanging the roles of $\chi^+$ and $\chi^-$ if necessary, we may assume that  $\chi^-(1)<\chi^+(1)$.
The above direct computation proves that
\begin{equation}\label{new}
\chi^-(1)=\frac{1}{2}\left(\frac{q^{2m}-q}{q-1}-q^m\right)\quad \textrm{and}\quad\chi^+(1)=\frac{1}{2}\left(\frac{q^{2m}-q}{q-1}+q^m\right)
\end{equation}
(compare \cite[Section~$1$]{L}).

Fix $\zeta\in \mathcal{S}$. We claim that $\zeta_P^G$ is irreducible. From Mackey's irreducibility Criterion~\cite[Proposition~23, Section~$7.3$]{Serre}, we need to show that for every $s\in G\setminus P$, we have $\zeta_{sPs^{-1}\cap P}\neq \zeta^s$, where  $\zeta^s$ is the character of $sPs^{-1}\cap P$ defined by $(\zeta^s)(x)=\zeta(s^{-1}xs)$ and, as usual, $\zeta_{sPs^{-1}\cap P}$ is the restriction of $\zeta$ to $sPs^{-1}\cap P$. Fix a monomorphism $\psi$ from $P/P'$ into $\mathbb C^*$.
Since $\zeta$ is a class function of $P$, we need to consider
only elements $s$ in distinct $(P,P)$-double cosets. These correspond to the $P$-orbits $\langle v\rangle$, $v^\perp\setminus\langle v\rangle$ and $V\setminus v^\perp$. Let $H=\langle v,u\rangle$ be a hyperbolic plane and choose $s\in G$ such that 
\[vs=u,\quad us=u\quad\textrm{and}\quad s_{H^\perp}=1_{H^\perp}.\]
A calculation shows that $\zeta^s(x)=\psi(\mu^{-1})=\overline{\zeta(x)}$, where $vx=\mu v$. Since $q-1$ is odd, we have $\zeta(x)\ne\zeta^s(x)$ when $\mu\ne1$. Therefore $\zeta\neq \zeta^s$.
Finally choose $s\in G$ such that $(v,u,w,z)s=(w,z,v,u)$, where $H=\langle v,u\rangle\perp \langle w,z\rangle$ is an orthogonal sum of hyperbolic planes 
and $s_{H^\perp}=1_{H^\perp}$.
Another calculation shows that  $\zeta^s(x)=\psi(\lambda)$ and $\zeta(x)=\psi(\mu)$, where $vx=\mu v$ and $wx=\lambda$. If $\mu\ne\lambda$, then $\zeta^s(x)\ne\zeta(x)$ and hence
$\zeta^s\ne\zeta$. Our  claim is now proved.

Write $\pi'=\sum_{i=1}^\ell m_i\chi_i$ as a linear combination of the distinct irreducible constituents of $\pi'$. Observe that, by the previous paragraph, each $\chi_i$ is of the form $\zeta_P^G$, for some $\zeta\in \mathcal{S}$. Therefore $\chi_i$ has degree $|G:P|$ for each $i$ and, in particular, $1$, $\chi^+$ and $\chi^-$ are not irreducible constituents of $\pi'$.

The number of irreducible constituents of $\pi_0$ is
\[1+1+1+2(m_1+\cdots +m_\ell)= 3+2|\mathcal{S}|=3+2\left(\frac{q}{2}-1\right)=q+1,\]
and by~\eqref{eq1} we have
\[3+4m_1^2+\cdots+4m_\ell^2=2q-1.\]
Multiplying the first equation by $-2$ and adding the second equation we have
 \[-3+4m_1(m_1-1)+\cdots+4m_\ell(m_\ell-1)=-3.\]
It follows  that $m_1=\cdots =m_\ell=1$, and hence $\ell=q/2-1$. This shows that $\pi'$ is multiplicity-free.

Summing up, we have
\begin{equation}\label{eq11}
\pi_0=1+\chi^++\chi^-+2\pi',\quad\langle\pi',\pi'\rangle=\frac{q}{2}-1,\quad\langle 1+\chi^++\chi^-,\pi'\rangle=0.
\end{equation}

We now turn our attention to the characters $\pi^+$ and $\pi^-$. By Frobenius reciprocity, or by~\cite[Theorem~$1$~(i) and~(ii)]{I}, we see that
\begin{equation}\label{eq2}
\langle \pi^+,\pi^+\rangle=\langle \pi^-,\pi^-\rangle=\frac{q}{2}+1.
\end{equation}

By~\cite[Lemma~$2$ (iii) and (iv)]{I}, the orbits of $\mathrm{O}_{2m}^-(q)$ in its action on $\Omega^+$ are in one-to-one correspondence with the  elements in $\{\alpha+\alpha^2\mid \alpha\in \mathbb{F}_q\}$. In particular, we have $\langle \pi^+|_{\mathrm{O}_{2m}^-(q)},1\rangle=|\{\alpha+\alpha^2\mid \alpha\in \mathbb{F}_q\}|=q/2$. Now Frobenius reciprocity implies that
\begin{equation}\label{eq3}
\langle \pi^+,\pi^-\rangle=\frac{q}{2}.
\end{equation}

Next we show that
\begin{equation}\label{eq4}
\langle \pi_0,\pi^+\rangle=\langle\pi_0,\pi^-\rangle=q.
\end{equation}
Using Frobenius reciprocity, it suffices to show that the number of orbits of $\mathrm{O}_{2m}^\pm(q)$ on $V\setminus\{0\}$ is $q$. Fix $\varepsilon\in \{+,-\}$ and  let $Q^\varepsilon~$ be the quadratic form on $V$ preserved by $\mathrm{O}_{2m}^\varepsilon(q)$. For $\lambda\in \mathbb{F}_q$, we see from~\cite[Lemma~$2.10.5$~(ii)]{KL} that $\Omega_{2m}^\varepsilon(q)$ is transitive on $V_\lambda^\varepsilon=\{v\in V\setminus\{0\}\mid Q^\varepsilon(v)=\lambda\}$. In particular, $\{V_\lambda^\varepsilon \mid \lambda\in \mathbb{F}_q \}$ is  the set of orbits of $\Omega_{2m}^\varepsilon(q)$ on $V\setminus\{0\}$. Since $\mathrm{O}_{2m}^\varepsilon(q)$ is the isometry group of $Q^\varepsilon$ we see that $\{V^{\epsilon}_\lambda \mid \lambda \in \mathbb{F}_{q}\}$ is also the set of orbits of $\mathrm{O}_{2m}^\varepsilon(q)$ on $V\setminus\{0\}$, and~\eqref{eq4} is now proved.

Since $\pi^+$ is multiplicity-free, up to reordering, by~\eqref{eq11} and~\eqref{eq4}, we may assume that
\begin{equation*}
\pi^+=1+a\chi^-+b\chi^++\sum_{i=1}^t\chi_i+\rho,
\end{equation*}
where $a,b\in \{0,1\}$, $0\le t\le\frac q2-1$ and $\langle \pi_0,\rho\rangle=0$.
By ~\eqref{eq4}, we have  $q-2\ge 2t=q-1-a-b\ge q-3$. Hence $2t=q-2$ and $\{a,b\}=\{0,1\}$.
Since $\pi^+(1)=|\Omega^+|=q^m(q^m+1)/2$ and $\pi'(1)=(q/2-1)|G:P|=(q/2-1)(q^{2m}-1)/(q-1)$, it follows by~\eqref{new} that $a=0$ and $b=1$.
By~\eqref{eq2}, we have $\pi^+=1+\chi^++\pi'$.

Now~\eqref{eq11},~\eqref{eq2},~\eqref{eq3} and~\eqref{eq4} imply immediately that $\pi^-=1+\chi^-+\pi'$. This shows that
\[\pi^++\pi^-=(1+\chi^++\pi')+(1+\chi^{-}+\pi')=1+1+\chi^++\chi^-+2\pi'=1+\pi_0=\pi,\]
which completes the proof of Theorem~\ref{thrm}.
\end{proof}

We note that the ``$q=2$'' case  of Theorem~\ref{thrm} was first proved by Siemons and Zalesskii in~\cite[Proposition~$3.1$]{SZ}. This case is particularly easy to deal with (considering that the action of $G$ on both $\Omega^+$ and $\Omega^-$ is  $2$-transitive) and its proof depends only on Frobenius reciprocity. However, the general statement (valid for every even $q$) of Theorem~\ref{thrm} was undoubtedly inspired by their observation.

Theorem~\ref{thrm} reproduces the following result as an immediate corollary (see \cite[Theorem 6]{D}).
\begin{corollary}Every element of $\Sp_{2m}(q)$ is conjugate to an element of $\mathrm{O}_{2m}^+(q)$ or of $\mathrm{O}_{2m}^-(q)$.
\end{corollary}
\begin{proof}
Let $g\in \Sp_{2m}(q)$. Then $\pi(g)=(1+\pi_0)(g)=1(g)+\pi_0(g)\geq 1(g)=1$ and therefore, since $\pi=\pi^{+} + \pi^{-}$ by Theorem~\ref{thrm}, either $\pi^+(g)\geq 1$ or $\pi^-(g)\geq 1$; that is, $g$ fixes some point in $\Omega^+$ or in $\Omega^-$. In the first case $g$ has a conjugate in $\mathrm{O}_{2m}^{+}(q)$ and in the second case $g$ has a conjugate in $\mathrm{O}_{2m}^-(q)$.
\end{proof}

\end{document}